\documentclass[11pt,reqno]{amsart}
\usepackage{amssymb,latexsym}
\usepackage{cite}
\usepackage{latexsym, amsmath, amssymb}
\theoremstyle{plain}
\newtheorem{theorem}{Theorem}
\newtheorem{lemma}{Lemma}
\newtheorem{corollary}{Corollary}

\theoremstyle{definition}
\newtheorem{definition}{Definition}
\theoremstyle{remark}
\newtheorem{remark}{Remark}
\newtheorem{example}{Example}
\numberwithin{equation}{section}
\makeatletter
\renewcommand{\@seccntformat}[1]{\csname the#1\endcsname.\hspace{1em}}
\makeatother
\def\DJ{\leavevmode\setbox0=\hbox{D}\kern0pt\rlap
 {\kern.04em\raise.188\ht0\hbox{-}}D}

\begin{document}

\title[Common fixed point theorems for hybrid F-contractions]
{Common fixed point theorems for Hybrid generalized $(F,\varphi)$-contractions
under common limit range property with applications}

\author[H. K. Nashine]{Hemant Kumar Nashine$^1$}
\address{$^1$Department of Mathematics, Amity School of Applied Sciences,
Amity University Chhattisgarh, Manth/Kharora (Opp.\ ITBP) SH-9,
Raipur, Chhattisgarh - 493225, India.}
\email{drhknashine@gmail.com}
\author[M. Imdad]{Mohammad Imdad$^{2}$}
\address{$^{2}$Department of Mathematics, Aligarh Muslim University,
 Aligarh, UP-202002, India.}
\email{mhimdad@yahoo.co.in}
\author[M. Ahmadullah]{Md Ahmadullah$^{3,*}$}
\address{$^{3,*}$Department of Mathematics,
Aligarh Muslim University,
Aligarh, UP-202002, India.}
\email{ahmadullah2201@gmail.com}
\address {$^*$ corresponding author}

\begin{abstract}
We consider a relatively new hybrid generalized $F$-contraction involving a pair of mappings
and utilize the same to prove a common fixed point theorem for a
hybrid pair of occasionally coincidentally idempotent mappings
satisfying generalized $(F,\varphi)$-contraction condition
under common limit range property in complete metric
spaces. A similar result involving a hybrid pair of mappings satisfying a
Rational type Hardy-Rogers $(F,\varphi)$-contractive condition is also proved.
Our results generalize and improve several results of the existing
 literature. As applications of our results, we prove two theorems for
 the existence of solutions of certain system of functional equations
arising in dynamic programming, and Volterra integral inclusion besides
providing an illustrative example.
\end{abstract}



\maketitle

\section{Introduction and preliminaries}
Let $(X,d)$ be a metric space. Then, following the Nadler \cite{Nadler}, we adopt the following notations:
\begin{itemize}
  \item $CL(X)=\{A: A$ is a non-empty closed subset of $X$\}.
  \item $CB(X)=\{A: A$ is a non-empty closed and bounded subset of $X$\}.
  \item For non-empty closed and bounded subsets $A,B$ of $X$ and $x\in X$,
 $$d(x,A)=\inf\big\{d(x,a):a\in A\big\}$$
  and
  $$\mathcal{H}(A,B)=\max\Big\{\sup\big\{d(a,B):a\in A\big\},\sup\big\{d(b,A):b\in B\big\}\Big\}.$$
\end{itemize}

Recall that $CB(X)$ is a metric space with the metric
$\mathcal{H}$ which is known as the Hausdorff-Pompeiu metric on
$CB(X)$.

 \vspace{.3cm}
In 1969, Nadler \cite{Nadler} proved that every multi-valued contraction mapping defined  on
a complete metric space has a fixed point.
In proving this result, Nadler used the idea of Hausdorff metric to establish the multi-valued version of Banach
Contraction Principle which runs as follows:

\begin{theorem}\label{nadler}
Let $(X,d)$ be a complete metric space and $\mathcal{T}$ a mapping from ${X}$ into $CB({X})$ such that
for all $x,y \in {X}$,
$$ \mathcal{H}(\mathcal{T}{x}, \mathcal{T}{y}) \leq {\lambda} d(x,y),$$
where $\lambda \in [0,1)$. Then $\mathcal{T}$ has a fixed point, $i.e.,$ there exists a point $x \in {X}$ such that $x \in
\mathcal{T}x$.
\end{theorem}

\vspace{.3cm}
Hybrid fixed point theory involving pairs of single-valued and
multi-valued mappings is a relatively new development in Nonlinear Analysis
(e.g.\cite{Fisher,Hu,IKS,Kubiak,Naimpally,SKI} and
references therein). The much discussed concepts of commutativity and weak
commutativity were extended to hybrid pair of mappings on metric
spaces by Kaneko~\cite{K1,K2}. In 1989, Singh et al.~\cite{SHC}
extended the notion of compatible mappings and obtained some
coincidence and common fixed point theorems for nonlinear hybrid
contractions. It was observed that under compatibility the fixed
point results usually require continuity of one of the underlying
mappings. Afterwards, Pathak \cite{Pathak} generalized the concept
of compatibility by defining weak compatibility for hybrid pairs
of mappings (including single valued case as well) and utilized the same to
prove common fixed point theorems. Naturally, compatible mappings
are weakly compatible but not conversely.

\vspace{.3cm}
In 2002, Aamri and El-Moutawakil \cite{Aamri} introduced (E.A) property for single-valued mappings.
Later,  Kamran \cite{Kamran} extended the notion of (E.A) property to hybrid pairs of mappings. In 2011,
Sintunavarat and Kumam \cite{Sintunavarat} introduced the notion
of common limit range (CLR) property for single-valued mappings and
showed its superiority over the property (E.A). Motivated by this
fact, Imdad et al.~\cite{ICSA} established common limit range
property for a hybrid pair of mappings and proved some fixed point
results in symmetric (semi-metric) spaces. For more details on hybrid contraction conditions, one can consult
\cite{Ali,CKI,Dhompongsa,IAK,Imdad3,KCI,Kaneko,Naimpally,PKC,PK,Singh,SM,SM1}.

\vspace{.3cm}
The following definitions and results are standard in the theory of hybrid pair of mappings.

\vspace{.3cm}
\begin{definition}\label{D0}
Let $f:X\rightarrow X$ and $T:X\rightarrow CB(X)$ be a single-valued and multi-valued mapping respectively. Then
\begin{itemize}
  \item A point $x\in X$ is a fixed point of $f$ (resp.\ $T$) if $x=fx$ (resp.\ $x\in Tx$).
The set of all fixed points of $f$ (resp.\ $T$) is denoted by $F(f)$ (resp. $F(T)$).
  \item A point $x\in X$ is a coincidence point of $f$ and $T$ if $fx\in Tx$.
The set of all coincidence points of $f$ and $T$ is denoted by $\mathcal{C}(f,T)$.
  \item A point $x\in X$ is a common fixed point of $f$ and $T$ if $x=fx\in Tx$.
The set of all common fixed points of $f$ and $T$ is denoted by $F(f,T)$.
  \item $T$ is a closed multi-valued mapping if the graph of $T~i.e.,
   ~G(T)=\{(x, y) : x \in X, y \in Tx\}$ is a closed subset of $X \times X$.
\end{itemize}
\end{definition}

We also recall the following terminology often used in the considerations of a hybrid
pairs of mappings.

\begin{definition}\label{D1}
Let $(X,d)$ be a metric space with $f:X\rightarrow X$ and $T:X\rightarrow CB(X)$. Then a hybrid pair of mappings $(f,T)$ is said to be:
\begin{itemize}
  \item commuting on $X$ \cite{K1} if $fTx\subseteq Tfx~~ \forall x\in X$.
  \item weakly commuting on $X$ \cite{K2} if $\mathcal{H}(fTx,Tfx)\leq d(fx,Tx) ~~\forall x\in X$.
  \item compatible \cite{SHC} if $fTx\in CB(X)~~\forall x\in X$ and $\lim\limits_{n\rightarrow\infty}\mathcal{H}(Tfx_{n},fTx_{n})=0$,
whenever $\{x_{n}\}$ is a sequence in $X$ such that $$\lim_{n\to \infty}Tx_{n}\rightarrow A\in CB(X)~ {\rm and}~ \lim_{n\to \infty}fx_{n}\rightarrow t\in A.$$
  \item non-compatible \cite{Kaneko} if $\exists$ at least one sequence $\{x_{n}\}$ in $X$ such that
$$\lim_{n\to \infty}Tx_{n}\rightarrow A\in CB(X)~{\rm  and}~ \lim_{n\to \infty}fx_{n}\rightarrow t\in A~{\rm but}~ \lim\limits_{n\rightarrow\infty}\mathcal{H}(Tfx_{n},fTx_{n})$$ is
either non-zero or nonexistent.
  \item weakly compatible \cite{JR3} if $Tfx= fTx$ for each $x\in \mathcal{C}(f,T)$.
  \item coincidentally idempotent \cite{IAK} if for every $v\in \mathcal{C}(f,T), ~ffv=fv~
  i.e., ~f$ is idempotent at the coincidence points of $f$ and $T$.
  \item occasionally coincidentally idempotent \cite{PL} if $ffv=fv$ for some $v\in \mathcal{C}(f,T)$.
  \item enjoy the property (E.A) \cite{Kamran} if $\exists$ a sequence $\{x_{n}\}$ in $X$ such that
  $$\lim_{n\rightarrow\infty}fx_{n}=t\in A=\lim_{n\rightarrow\infty}Tx_{n},$$
for some $t\in X$ and $A\in CB(X)$.
 \item enjoy common limit range property with respect to the mapping
$f$ (in short $CLR_{f}$ property) \cite{ICSA} if $\exists$ a sequence $\{x_{n}\}$ in $X$ such
that $$\lim_{n\rightarrow\infty}fx_{n}=fu\in A=\lim_{n\rightarrow\infty}Tx_{n},$$
for some $u\in X$ and $A\in CB(X)$.
\end{itemize}
\end{definition}

The following example demonstrates the interplay of the
occasionally coincidentally idempotent property with other notions described in the preceding definition.

\begin{example}\label{ex1}\cite[Example 1]{KCI}
Let $X=\{1,2,3\}$ (with the standard metric),
\[
f:\begin{pmatrix}1 & 2 & 3\\ 1 & 3 & 2
\end{pmatrix}\quad\text{and}\quad T:\begin{pmatrix}1 & 2 & 3\\
\{1\} & \{1,3\} & \{1,3\} \end{pmatrix}.
\]
Then, it is straight forward to observe the following:
\begin{itemize}
\item $\mathcal{C}(f,T)=\{1,2\}$, and $F(f,T)=\{1\}$.
\item $(f,T)$ is not commuting and not weakly commuting.
\item $(f,T)$ is not compatible.
\item $(f,T)$ is not weakly compatible.
\item $(f,T)$ is not coincidentally idempotent since $ff2=f3=2\ne3=f2$.
\item $(f,T)$ is occasionally coincidentally idempotent since $ff1=1=f1$.
\end{itemize}
Obviously, in this case $(f,T)$ is also non-compatible, but simple
modifications of this example can show that the occasionally
coincidentally idempotent property is independent of this notion,
too.
\end{example}

The following example (taken from \cite{KCI}) demonstrates
the relationship between (E.A) property and common limit range
property.

\begin{example}\label{example1}\cite[Examples 2 and 3]{KCI}
Let $X=[0,2]$ be a metric space equipped with the usual metric $d(x,y)=\vert
x-y\vert$. Define $f,g:X\rightarrow X$ and $T:X\rightarrow CB(X)$
as follows:
\begin{equation*}
fx=\begin{cases}
               2-x, & \hbox{if $0\leq x<1$,} \\
               \frac{9}{5}, & \hbox{if $1\leq x\leq2$;}
    \end{cases}
\quad gx=\begin{cases}
               2-x, & \hbox{if $0\leq x\leq1$,} \\
               \frac{9}{5}, & \hbox{if $1<x\leq2$;}
    \end{cases}
\quad
Tx=\begin{cases}
               \left[\frac{1}{2},\frac{3}{2}\right], & \hbox{if $0\leq x\leq1$,} \\
               \left[\frac{1}{4},\frac{1}{2}\right], & \hbox{if $1<x\leq2$.}
    \end{cases}
\end{equation*}

\vspace{.3cm} One can verify that the pair $(f,T)$ enjoys the property (E.A),
but not the $CLR_{f}$ property. On the other hand, the pair $(g,T)$ satisfies the
$CLR_{g}$ property.
\end{example}

\begin{remark}\label{remark1}
If a pair $(f,T)$ satisfies the property (E.A) along with the closedness of $f(X)$,
then the pair also satisfies the $CLR_f$ property.
\end{remark}

Throughout this paper, we denote by $\mathbb{R}$ the set of all
real numbers, by $\mathbb{R}^+$ the set of all positive real
numbers and by $\mathbb{N}$ the set of all positive integers. In
what follows, $\mathcal{F}$ denote the family of all
functions $F : \mathbb{R}^+ \to \mathbb{R}$ that satisfy the
following conditions:

\begin{itemize}

\item[(F$_1)$] $F$ is continuous and strictly increasing;

\item[(F$_2)$] for each sequence $\{\beta_n\}$ of positive numbers,
$\displaystyle\lim_{n \to \infty} \beta_n = 0~\iff~\displaystyle\lim_{n \to
\infty} F(\beta_n) = -\infty$;

\item[(F$_3$)] there exists $k \in (0, 1)$ such that
 $\displaystyle\lim_{\beta\to 0^+} \beta^k F(\beta) = 0$.
\end{itemize}

\vspace{.2cm}
Some examples of functions $F\in\mathcal{F}$ are $F(t)=\ln t$,
$F(t)=t+\ln t$, $F(t)=-1/\sqrt t$, see~\cite{Wardowski}.

\vspace{.3cm}
\begin{definition} \cite{Wardowski}
Let $(X, d)$ be a metric space. A self-mapping $T$ on $X$ is called an $F$-contraction if there exist
$F \in \mathcal{F}$ and $\tau \in \mathbb{R}^+$ such that
\begin{equation}\label{equation01}
\tau + F(d(Tx, Ty)) \leq F(d(x, y)),
\end{equation}
for all $x, y \in X$ with $d(Tx, Ty) > 0$.
\end{definition}

\begin{example}\cite{Wardowski} Let $F : \mathbb{R}^+ \to
\mathbb{R}$ be a mapping given by $F(x) = \ln x$. It is clear that $F$
satisfies (F$_1$)--(F$_3$) for any $k \in (0, 1)$. Under this setting, (\ref{equation01}) reduces to
$$d(Tx, Ty) \leq e^{-\tau} d(x, y), \text{ for all } x, y \in X,
\quad Tx \neq Ty.$$

Notice that for $x, y \in X$ such that $Tx = Ty$, the previous
inequality also holds and hence $T$ is a contraction.
\end{example}

In what follows, for a metric space $(X,d)$ and a multi-valued
mapping $T:X\to CL(X)$, we denote
\[
M(x,y)=\max \left\{d(x, y), d(x, Tx), d(y, Ty), \frac{1}{2}\big[d(x,
Ty) + d(y, Tx) \big] \right\}.
\]

\begin{definition}\cite{Sgroi}
Let $(X, d)$ be a metric space. A multi-valued mapping $T: X \to
CL(X)$ is called an $F$-contraction if there exist $F \in
\mathcal{F}$ and $\tau \in \mathbb{R}^+$ such that for all $x, y
\in X$ with $y \in Tx,~~ \exists z \in Ty$,
\begin{equation}\label{equation02}
\tau + F(d(y, z)) \leq F \left(M(x,y)\right), ~~{\rm whenever}~ d(y, z) > 0.
\end{equation}

\end{definition}

\begin{example}\cite{Sgroi}
Let $F : \mathbb{R}^+ \to \mathbb{R}$ be mapping given by $F(x) = \ln x$. Then
for each multi-valued mapping $T: X \to CL(X)$ satisfying
(\ref{equation02}), we have
$$
d(y, z) \leq e^{-\tau} M(x, y), \text{ for all } x, y \in X,\quad
z\in Ty, \ y\neq z.
$$
It is clear that for $z, y \in X$ such that $y = z$ the previous
inequality also holds.
\end{example}

Some fixed point results for single-valued (resp. multi-valued)
$F$-contractions were obtained in
\cite{Wardowski,Alsulami,Karapinar} (resp. \cite{Sgroi}).

\vspace{.3cm}
Our aim in this paper is to prove a common fixed point theorem
for a hybrid pair of occasionally coincidentally idempotent
mappings satisfying generalized $(F,\varphi)$-contraction condition,
under $CLR$ property in complete metric spaces. Also, a similar result for a variant of
Rational type Hardy-Rogers generalized $(F,\varphi)$-contractive condition is also derived. Here, it can be pointed out that Sgroi and Vetro~\cite{Sgroi} introduced and studied
such conditions for multi-valued mappings while the similar conditions were earlier introduced an studied
by Wardowski~\cite{Wardowski} for single-valued mappings. Our results generalize and improve
several known results of the existing literature. Finally, we utilize our results to prove the existence
of solutions of certain system of functional equations
arising in dynamic programming, as well as Volterra integral inclusion besides providing an illustrative example.

\vspace{.5cm}
\section{The Main Results}

This section is divided into two parts. In the first subsection,
we prove a common fixed point theorem for a hybrid pair of
occasionally coincidentally idempotent mappings satisfying a generalized $(F,\varphi)$-contractions condition via $CLR$
property in complete metric spaces, while in the
second one we obtain results for hybrid pairs which satisfy a
Rational Hardy-Rogers type $(F,\varphi)$-contractive condition.

\begin{definition}\label{MD1}
  Let $(X,d)$ be a metric space, $f: X \to X$ and $T : X \to CB(X)$.
  Then hybrid pair $(f,T)$ is said to be a generalized $(F,\varphi)$-contraction, if  there exist an increasing, upper semicontinuous mapping from the right
  $$
\Phi= \big\{\varphi: [0,\infty) \to [0,\infty)~| ~\limsup_{ s\to t^+}\varphi(s)<\varphi(t),~~ \varphi(t) <t, \forall ~, t > 0 \big\},
$$
$F \in \mathcal{F}$ and $\tau \in
\mathbb{R}^+$ such that
{\small\begin{align}\label{equation1}
\nonumber \tau &+ F(\mathcal{H}^p(Tx,Ty))\\
& \leq F\bigg( \varphi\bigg(\max \bigg\{\begin{array}{cc}
d^p(f x, T x), d^p(f y, T y), d^p(f y, f x),\frac{1}{2}\big[d^p(f x,T y) + d^p(f y, T x)\big], \\
\frac{d^p(f x, T x)d^p(f y, T y)}{1 + d^p(f y, fx)}, \frac{d^p(f x,T y)d^p(f y, Tx)}{1 + d^p(f y, f x)},
\frac{d^p(f x,T y)d^p(f y,Tx)}{1 + d^p(Tx, Ty)}
  \end{array}
  \bigg\} \bigg) \bigg)
\end{align}}
for all $x,y\in X$, $p\geq 1$ with $\mathcal{H}(Tx,Ty)>0.$
\end{definition}

\begin{definition}\label{MD2}
  Let $(X,d)$ be a metric space, $f: X \to X$ and $T : X \to CB(X)$. Then hybrid pair $(f,T)$ is said to be a Rational Hardy-Rogers $(F,\varphi)$-contraction, if  there exist an increasing,
upper semicontinuous mapping from the right
$$\Phi= \{ \varphi : [0,\infty) \to [0,\infty) ~|~ \limsup_{ s\to t^+}
\varphi(s) < \varphi(t),~ \varphi(t) < t, \forall \, t > 0 \},$$
$F \in \mathcal{F}$ and $\tau \in \mathbb{R}^+$ such that
{\small \begin{align}\label{equation2}
\nonumber \tau &+ F(\mathcal{H}^p(Tx,Ty))\\
& \leq F\bigg(\varphi\bigg(\begin{array}{cc}
{ \alpha} d^p(fx,fy) + \frac{ { \beta} \big[1+d^p(fx,Tx)\big]d^p(fy,Ty) }{1+ d^p(fx,fy)}+ {\gamma} \big[d^p(fx,Tx)+ d^p(fy,Ty)\big]\\
+\delta \big[{d^p(fx,Ty)} + d^p(fy,Tx)\big]
\end{array}
 \bigg) \bigg)
\end{align}}
for all $x, y \in X$  with $Tx \neq Ty $, where $p\geq 1$, $\alpha, \beta,
\gamma, \delta \geq 0$, $\alpha + \beta + 2\gamma +2 \delta \leq 1$.
\end{definition}

\vspace{.3cm}
Now we propose our first main result as follows:

\begin{theorem}\label{theorem1}
Let $(X,d)$ be a metric space,  $f: X \to X$ and $T : X \to
CB(X)$.  If the hybrid pair $(f,T)$ satisfies generalized
 $(F,\varphi)$-contraction condition $(\ref{equation1})$,
and also enjoys the $CLR_{f}$ property, then the mappings $f$ and $T$ have a
coincidence point.

Moreover, if the hybrid pair $(f,T)$ is occasionally coincidentally
idempotent, then the pair $(f,T)$ has a common fixed point.
\end{theorem}

\begin{proof} Since the pair $(f,T)$ enjoys the $CLR_f$ property, there exists a sequence $\{x_{n}\}$ in $X$ such that
$$\lim_{n\rightarrow\infty}fx_{n}=fu\in A=\lim_{n\rightarrow\infty}Tx_{n},$$
for some $u\in X$ and $A\in CB(X)$. We assert that $fu\in Tu$. If
not, then using condition \eqref{equation1}, we have
{\small \begin{eqnarray*}
\nonumber
\tau &+& F\big(\mathcal{H}^p(Tx_{n},Tu)\big)\\
&\leq& F\bigg( \varphi\bigg(\max \bigg\{\begin{array}{cc}
d^p(fx_{n},Tx_{n}),d^p(fu,Tu),d^p(fu, fx_{n}),
\frac{1}{2}\big[d^p(fx_{n},Tu)+d^p(fu,Tx_{n})\big],\\
\frac{d^p(fx_{n},Tx_{n})d^p(fu, Tu)}{1 + d^p(fu, fx_{n})},
\frac{d^p(fx_{n},Tu)d^p(fu, Tx_{n})}{1 + d^p(fu, fx_{n})},
 \frac{d^p(fx_{n},Tu)d^p(fu, Tx_{n})}{1 + d^p(Tx_{n}, Tu)}
\end{array}
 \bigg\} \bigg) \bigg).
\end{eqnarray*}}
Passing to the limit as $n\rightarrow\infty$, we have
\begin{align*}
\tau + F(&\mathcal{H}^p(A,Tu))\\
&\leq F\left( \varphi\left(\max\left\{ \begin{array}{cc}
d^p(fu,A),d^p(fu,Tu),0,\frac{1}{2}\big[d^p(fu,Tu)+d^p(fu,A)\big],\\
\frac{d^p(fu, A)d^p(fu, Tu)}{1 + d^p(fu, fu)}, \frac{d^p(fu, Tu)d^p(fu, A)}{1 + d^p(fu, fu)},
\frac{d^p(fu, Tu)d^p(fu, A)}{1 + d^p(A, Tu)}
\end{array}
 \right\} \right) \right)\\
&=F\left(\varphi\left(\max\left\{ \begin{array}{cc}
d^p(fu,A),d^p(fu,Tu),0,\frac12\big[d^p(fu,Tu)+d^p(fu,A)\big],\\
d^p(fu, A)d^p(fu, Tu), \frac{d^p(fu, A)d^p(fu, Tu)}{1 + d^p(A, Tu)}
\end{array}
 \right\} \right) \right).
\end{align*}
Using  $fu\in A, \tau>0$, ($F_1$) and property of $\Phi$, we have
\begin{eqnarray*}
\nonumber
\mathcal{H}^p(A,Tu) &\leq& \varphi\Big(\max\big\{0,d^p(fu,Tu),0,\frac{1}{2}[d^p(fu,Tu)+0],0,0\big\}\Big)\\
&=&\varphi\big(d^p(fu,Tu)\big)\\
&<&d^p(fu,Tu).
\end{eqnarray*}
Since $fu\in A$ the above inequality implies
$$d(fu,Tu) \leq \mathcal{H}(A,Tu)< d(fu,Tu),$$
a contradiction. Hence $fu\in Tu$ which shows that the
pair $(f,T)$ has a coincidence point ($i.e.,~
\mathcal{C}(f,T)\neq\emptyset$).

\vspace{.3cm}
Now, assume that the hybrid pair $(f,T)$ is occasionally
coincidentally idempotent. Then for some $v\in
\mathcal{C}(f,T), {\rm we~ have }~ffv=fv\in Tv$. Our claim is that
$Tv=Tfv$. If not, then using condition \eqref{equation1}, we get
{\small \begin{align*}
\tau + F(&\mathcal{H}^p(Tfv,Tv)) \\
&\leq F
\bigg( \varphi\bigg(\max \bigg\{\begin{array}{cc}
 d^p(ffv,Tfv), d^p(fv,Tv), d^p(fv, ffv), \frac{1}{2}\left[d^p(fv,Tfv)+d^p(ffv,Tv)\right],\\
\frac{d^p(ffv, Tfv)d^p(fv, Tv)}{1 + d^p(fv, ffv)}, \frac{d^p(fv, Tfv)d^p(ffv, Tv)}{1 + d^p(fv, ffv)},
\frac{d^p(fv, Tfv)d^p(ffv, Tv)}{1 + d^p(Tfv, Tv)}
  \end{array}
  \bigg\} \bigg) \bigg)\\
  &\\
&= F \bigg( \varphi\bigg(\max \bigg\{\begin{array}{cc}
d^p(fv,Tfv), d^p(fv,Tv),0,\frac{1}{2}\left[d^p(fv,Tfv)+d^p(fv,Tv)\right],\\
d^p(fv, Tfv)d^p(fv, Tv), d^p(fv, Tfv)d^p(fv, Tv),
\frac{d^p(fv, Tfv)d^p(fv, Tv)}{1 + d^p(Tfv, Tv)}
\end{array}
\bigg\} \bigg) \bigg).
\end{align*}}
Since $fv\in Tv$, the above inequality implies
\begin{align*}
\tau + F(\mathcal{H}^p(Tfv,Tv)) &\leq F \Big(\varphi\Big(\max\Big\{d^p(fv,Tfv),0,0,\frac{1}{2}d^p(fv,Tfv),0,0,0\Big\}\Big) \Big)\\
    &= F \Big(\varphi\big(d^p(Tfv,fv)\big)\Big).
\end{align*}
Using $(F_1)$ and property of $\Phi$, we get
$$d^p(Tfv,fv) < d^p(Tfv,fv),$$
which is a contradiction. Thus we have $fv=ffv\in Tv=Tfv$ which
shows that $fv$ is a common fixed point of the mappings $f$ and $T$.
\end{proof}

In view of $Remark$ \ref{remark1}, we have the following natural
result:

\begin{corollary}\label{corollary1}
Let $(X,d)$ be a metric space,  $f: X \to X$ and $T : X \to
CB(X)$.  If the hybrid pair $(f,T)$ satisfies generalized
 $(F,\varphi)$-contraction condition $(\ref{equation1})$, and enjoys the (E.A) property along with the closedness of
$f(X)$, then the mappings $f$ and $T$ have a coincidence point.

Moreover, if the hybrid pair $(f,T)$ is occasionally coincidentally
idempotent, then the pair $(f,T)$ has a common fixed point.
\end{corollary}

Notice that, a non-compatible hybrid pair always satisfies the property (E.A). Hence, we get the following corollary:

\begin{corollary}\label{corollary2}
Let $f$ be a self mapping on a metric space $(X,d)$, $T$ a
mapping from $X$ into $CB(X)$ satisfying generalized
 $(F,\varphi)$-contraction condition $(\ref{equation1})$. If the hybrid pair
$(f,T)$ is non-compatible and $f(X)$ a closed subset of $X$,
then the mappings $f$ and $T$ have a coincidence point.

Moreover, if the pair $(f,T)$ is occasionally coincidentally idempotent,
then the pair $(f,T)$ has a common fixed point.
\end{corollary}

If $F : \mathbb{R}^+ \to \mathbb{R}$ is defoned by $F(t) = \ln t$
and denoting $e^{-\tau}=k$, then we have the following corollary:

\begin{corollary}\label{C-1}
Let $(X,d)$ be a metric space, $f: X \to X$ and $T : X \to CB(X)$.
Assume that there exist $k\in(0,1)$, $\varphi \in \Phi$ such that
{\small \begin{align*}
\mathcal{H}^p(Tx,Ty) \leq k
\varphi\left(\max \left\{\begin{array}{cc}
d^p(f x, T x), d^p(f y, T y), d^p(f y, f x),\frac{1}{2}\big[d^p(f x,T y) + d^p(f y, T x)\big], \\
\frac{d^p(f x, T x)d^p(f y, T y)}{1 + d^p(f y, fx)}, \frac{d^p(f x,T y)d^p(f y, Tx)}{1 + d^p(f y, f x)},
\frac{d^p(f x,T y)d^p(f y,Tx)}{1 + d^p(Tx, Ty)}
  \end{array}
  \right\} \right)
\end{align*}}
for all $x,y\in X$ with $\mathcal{H}(Tx,Ty)>0,~ p\geq 1$,
and the hybrid pair $(f,T)$ enjoys the $CLR_f$. Then the mappings $f$
and $T$ have a coincidence point.

Moreover, if the hybrid pair $(f,T)$ is occasionally coincidentally
idempotent, then the pair $(f,T)$ has a common fixed point.
\end{corollary}

\vspace{.4cm}
Now, we present our second main result as follows:

\vspace{.3cm}
\begin{theorem}\label{theorem2}
Let $(X,d)$ be a metric space, $f: X \to X$ and $T : X \to CB(X)$.
If the hybrid pair $(f,T)$ satisfies a Rational Hardy-Rogers $(F,\varphi)$-contraction
condition $(\ref{equation2})$ and also enjoys the $CLR_{f}$ property, then the mappings $f$ and $T$ have a coincidence point.

Moreover, if the hybrid pair $(f,T)$ is occasionally coincidentally
idempotent, then the pair $(f,T)$ has a common fixed point.
\end{theorem}

\begin{proof}
As the pair $(f,T)$ shares the $CLR_{f}$ property, there exists a sequence $\{x_{n}\}$ in $X$ such that
\begin{equation*}
  \lim_{n\rightarrow\infty}fx_{n}=fu\in A=\lim_{n\rightarrow\infty}Tx_{n},
\end{equation*}
for some $u\in X$ and $A\in CB(X)$. We assert that $fu\in Tu$. If
not, then using condition \eqref{equation2}, we have
{\small\begin{align*}
\tau +  F(\mathcal{H}^p(Tx_{n},Tu))\leq F
\left(\varphi\left(\begin{array}{cc}
 \alpha \, d^p(fx_{n},fu) +
\frac{ { \beta} [1+d^p(fx_{n},Tx_{n})]d^p(fu,Tu) }{1+ d^p(fx_{n},fu)}\\
+\gamma [d^p(fx_{n},Tx_{n}) +  d^p(fu,Tu)] + \delta [ d^p(fx_{n},Tu)+ d^p(fu,Tx_{n})]
\end{array}
 \right) \right).
\end{align*}}
Passing to the limit as $n\rightarrow\infty$ in the above inequality,
we obtain
$$
\tau + F(\mathcal{H}^p(A,Tu))\leq F(\varphi(\beta+\gamma+\delta)d^p(fu,Tu)),
$$
Using $\tau>0$ and ($F_1$) and property of $\Phi$, it follows that
$$
d^p(fu,Tu)\leq d^p(A,Tu)<(\beta+\gamma+\delta)d^p(fu,Tu),
$$
a contradiction, as $\beta+\gamma+\delta\leq1$.
Hence, $fu\in Tu$ which shows that the hybrid pair $(f,T)$ has a coincidence
point ($i.e., ~\mathcal{C}(f,T)\neq\emptyset$).

\vspace{.3cm}
Now, if the mappings $f$ and $T$ are occasionally
coincidentally idempotent, then there exists $v\in \mathcal{C}(f,T)$ such that $ffv=fv\in Tv$.
Our claim is that $fu$ is the common fixed point of $f$ and $T$. It is sufficient to show that
$Tv=Tfv$. If not, then using condition \eqref{equation2}, we have
{\small \begin{align*}
\tau + F(\mathcal{H}^p(Tfv,Tv)) &\leq F
\left(\varphi\left(\begin{array}{cc}
\alpha \, d^p(ffv,fv) + \frac{ { \beta} [1+ d^p(ffv,Tfv)]d^p(fv,Tv) }{1+ d^p(ffv,fv)}\\
+ \gamma [ d^p(ffv,Tfv) +  d^p(fv,Tv)] + \delta [ d^p(ffv,Tv)+ d^p(fv,Tfv)]
\end{array}
 \right) \right)\\
&= F \left(\varphi\left(\begin{array}{cc}
\beta [1+ d^p(fv,Tfv)]d^p(fv,Tv) + \gamma [ d^p(fv,Tfv) + d^p(fv,Tv)]\\
+ \delta [ d^p(fv,Tv)+ d^p(fv,Tfv)]
\end{array}
 \right) \right).
\end{align*}}
Since $fv\in Tv$, the above inequality implies
\begin{align*}
\tau + F\big(d^p(Tfv,Tv)\big) &\leq F\big(\varphi(\gamma+\delta) d^p(fv,Tfv)\big).
\end{align*}
Using $(F_1)$ and property of $\Phi$, we can have
$$
d^p(Tfv,fv) < (\gamma+\delta) d^p(fv,Tfv),
$$
a contradiction, as $ \gamma+\delta \leq1 $. Thus, $fv=ffv\in Tv=Tfv$ which shows that $fv$ is a common fixed
point of the mappings $f$ and $T$.
\end{proof}

\vspace{.3cm}
If $F : \mathbb{R}^+ \to \mathbb{R}$ is given by $F(t) = \ln t$
and denoting $e^{-\tau}=k$, then we have the following corollary:
\begin{corollary}\label{C-2}
Let $(X,d)$ be a metric space,  $f: X \to X$ and $T : X \to
CB(X)$. Suppose that there exist $k\in(0,1)$, $\varphi \in \Phi$ such that
\begin{align*}
\mathcal{H}^p(Tx,Ty)  & \leq
k \varphi\left(\begin{array}{cc}
{ \alpha} d^p(fx,fy) + \frac{ { \beta} [1+d^p(fx,Tx)]d^p(fy,Ty) }{1+ d^p(fx,fy)}+ {\gamma} [d^p(fx,Tx)+ d^p(fy,Ty)]\\
+\delta [{d^p(fx,Ty)} + d^p(fy,Tx)]
  \end{array}
\right)
\end{align*}
for all $x, y \in X$  with $Tx \neq Ty $, where  $p\geq 1,~\alpha, \beta,
\gamma, \delta \geq 0$, $\alpha + \beta + 2\gamma + 2\delta \leq 1$, and the hybrid pair $(f,T)$ enjoys
the $CLR_{f}$.
Then the mappings $f$ and $T$ have a coincidence point.

Moreover, if the pair $(f,T)$ is occasionally coincidentally
idempotent, then the pair $(f,T)$ has a common fixed point.
\end{corollary}

\vspace{.3cm}
\section{Illustrative Example}
 In this section, we provide an example to establish the genuineness of our extension.
\begin{example}\label{exmp}
  Let $X=[0,3]$ be a metric space equipped with the metric $d(x,y)=\vert
x-y\vert$. Define $f:X\rightarrow X$ and $T:X\rightarrow CB(X)$
as follows:
\begin{equation*}
fx=\begin{cases}
               3-x, & \hbox{if $x\in[0,2]$,} \\
               3, & \hbox{if $x\in(2,3]$;}
    \end{cases}
\quad
Tx=\begin{cases}
               \left[1,2\right], & \hbox{if $x\in[0,2]$,} \\
               \left[0,\frac{1}{2}\right], & \hbox{if $x\in(2,3]$.}
    \end{cases}
\end{equation*}
\end{example}
Let $F : \mathbb{R}^+ \to \mathbb{R}$ such that $F(t)=\ln(t)$, $\varphi : \mathbb{R}^+ \to \mathbb{R}^+$ such that $\varphi(t)=\frac{9}{10}t$ and $\tau=\frac{1}{5}>0$ and $p\geq 1$. Then it is easy to verify that
\begin{itemize}
  \item $F\in \mathcal{F}; ~\varphi \in \Phi; ~f(X)=[1,3]\cup\{3\}$, a closed set in $X;~\mathcal{C}(f,F)=[1,2]$;
  \item the hybrid pair $(f,T)$ satisfies $CLR_f$ property, as for the sequence $\{1+\frac{1}{n}\}_{n\in\mathbb{N}},$
      $$ \displaystyle\lim_{n\to\infty}f{x_{n}}=\lim_{n\to\infty}({2-\frac{1}{n}})
      =2=f1\in[1,2]=\lim_{n\to\infty}T({1+\frac{1}{n}});$$
\item $(f,T)$ is not coincidentally idempotent because $ff1=f2=1\ne 2=f1$;
\item $(f,T)$ is occasionally coincidentally idempotent, because $ff\frac{3}{2}=f\frac{3}{2}=\frac{3}{2}$;
\end{itemize}

\vspace{.3cm}
Now, in order to verify condition (\ref{equation1}), we distinguish two cases:

\vspace{.3cm}
\noindent Case I: If $x\in [0,2] ~{\rm and}~ y\in (2,3]$, then
\begin{eqnarray*}
\nonumber
\mathcal{H}(Tx,Ty) = \mathcal{H}\Big([1,2],\big[0,\frac{1}{2}\big]\Big)&=&max\Big\{d\big([1,2],\big[0,\frac{1}{2}\big]\big),~d\big([0,\frac{1}{2}],[1,2]\big)\Big\} \\
 &=&  max\Big\{\frac{3}{2}, 1\Big\}=\frac{3}{2},
\end{eqnarray*}
and $d(fy,Ty)=d\big(3,\big[0,\frac{1}{2}\big]\big)=\frac{5}{2}.$
Therefore, for all $p\geq 1$, we get
$$\mathcal{H}^{p}(Tx,Ty) = \Big(\frac{3}{2}\Big)^{p} < e^{-\frac{1}{5}}\frac{9}{10}\Big({\frac{5}{2}}\Big)^{p} =e^{-\frac{1}{5}}\Big(\frac{9}{10}d^{p}(fy,Ty)\Big)=e^{-\tau}\varphi\big(d^{p}(fy,Ty)\big).$$
Taking the logarithms of the both side along with $F(t)=ln(t)$, we get
$$\tau+F\big(\mathcal{H}^{p}(Tx,Ty)\big)<F\big(\varphi\big(d^{p}(fy,Ty)\big)\big).$$

\noindent Case II: If $x\in (2,3] ~{\rm and}~ y\in [1,2]$, then
$$\mathcal{H}(Tx,Ty) =\mathcal{H}\Big(\big[0,\frac{1}{2}\big],[1,2]\Big) = \frac{3}{2}~ {\text {and}}~ d(fx,Tx)=d\big(3,\big[0,\frac{1}{2}\big]\big)=\frac{5}{2}.$$
Therefore, for all $p\geq 1$, we get
$$\mathcal{H}^{p}(Tx,Ty) = \Big(\frac{3}{2}\Big)^{p} < e^{-\frac{1}{5}}\frac{9}{10}\Big({\frac{5}{2}}\Big)^{p} =e^{-\frac{1}{5}}\Big(\frac{9}{10}d^{p}(fx,Tx)\Big)=e^{-\tau}\varphi\big(d^{p}(fx,Tx)\big).$$
Taking the logarithm of both side of the above inequality and using $ln(t)=F(t)$, we get
$$\tau+F\big(\mathcal{H}^{p}(Tx,Ty)\big)<F\big(\varphi\big(d^{p}(fx,Tx)\big)\big).$$

\vspace{.2cm}\noindent
Notice that for $x,y\in [1,2]$ (or $x,y\in (2,3])~ \mathcal{H}(Tx,Ty)=0.$

\vspace{.2cm}
Thus, all the hypotheses of Theorem \ref{theorem1} are satisfied and the hybrid pair $(f,T)$ has the common fixed point (namely $\frac{3}{2}$).

\vspace{.3cm}
With a view to establish genuineness of our extension, notice that for $x=1, y=3$, we have
$$\mathcal{H}(Tx,Ty) = \frac{3}{2};~ d(fx,fy)=d(2,3)=1;$$
$$\frac{1}{2}\big[d(fx,Tx)+d(fy,Ty)\big]
=\frac{1}{2}\big[d(2,[1,2])+d(3,\big[0,\frac{1}{2}\big])\big]
=\frac{1}{2}(0+\frac{5}{2})=\frac{5}{4}~{\rm and}$$
$$\frac{1}{2}\big[d(fx,Ty)+d(fy,Tx)\big]
=\frac{1}{2}\big[d(2,\big[0,\frac{1}{2}\big])+d(3,[1,2])\big]
=\frac{1}{2}(\frac{3}{2}+1)=\frac{5}{4},$$

\vspace{.3cm}\noindent
which shows that the contractive condition of Theorem 11 (due to Kadelburg et al. \cite{KCI})
is not satisfied. Thus, in all our Theorem \ref{theorem1} is applicable to the present example while
Theorem 11 of Kadelburg et al. \cite{KCI} is not which substantiates the utility of Theorem \ref{theorem1}.

\vspace{.5cm}
\section{Applications}
As applications of our main results, we prove an existence theorem on bounded solutions of
a system of functional equations. Also, an existence theorem on the solution of integral inclusion is proved.

\vspace{.4cm}
\subsection{Application To Dynamic Programming}
In 1978,  Bellman and Lee \cite{Bell1}
first studied the existence of solutions for
functional equations wherein authors notice that the basic form of
functional equations in dynamic programming can be described as follows:
$$
q(x)=\sup_{y\in D}\{G(x,y,q(\tau (x,y)))\},\quad x\in W,
$$
where $\tau :W\times D\rightarrow W$, $G:W\times D\times
\mathbb{R} \rightarrow \mathbb{R}$ are mappings, while $W\subseteq
U$ is a state space, $D\subseteq V$ is a decision space, and $U$,
$V$ are Banach spaces.

\vspace{.3cm}
In $1984$, Bhakta and Mitra \cite{BM} obtained some existence
theorems for the following functional equation which arises in
multistage decision process related to dynamic programming
\begin{equation*}
q(x)=\sup_{y\in D}\big\{g(x,y)+G(x,y,q(\tau(x,y)))\big\},~ x\in W,
\end{equation*}
where $\tau :W\times D\rightarrow W$, $g:W\times D\rightarrow
\mathbb{R}$, $G:W\times D\times  \mathbb{R} \rightarrow
\mathbb{R}$ are mappings, while $W\subseteq U$ is a state space,
$D\subseteq V$ is a decision space, and $U$, $V$ are Banach
spaces.

\vspace{.2cm}
In recent years, a lot of work have been done in this direction wherein a multitude of
existence and uniqueness results have been obtained for solutions
and common solutions of some functional equations, including
systems of functional equations in dynamic programming using suitable fixed point results.
For more details one can consults \cite{LIU1,LIU2,HKP1,HKP2,HKP3,HKP4} and the references therein.

\vspace{.3cm}
Consider now a multistage process, reduced to the system of
functional equations
\begin{equation}\label{fe}
q_{i}(x)=\sup_{y\in D}\big\{g(x,y)+G_{i}(x,y,q_{i}(\tau(x,y)))\big\}, ~x\in W,~ i \in \{1,2\},
\end{equation}
where $\tau :W\times D\rightarrow W$, $g:W\times D\rightarrow
\mathbb{R}$, $G_{i}:W\times D\times  \mathbb{R} \rightarrow
\mathbb{R}$ are given mappings, while $W\subseteq U$ is a state
space, $D\subseteq V$ is a decision space, and $U$, $V$ are Banach
spaces. The purpose of this section is to prove the
existence of solutions for a system of functional equations
\eqref{fe} using Theorem~\ref{theorem1}.

\vspace{.3cm}
Let $B(W)$ be the set of all bounded real-valued functions on $W$.
 For an arbitrary $h\in B(W)$ define $\left\Vert h\right\Vert
=\displaystyle \sup_{x\in W}\left\vert h(x)\right\vert $, with
respective metric $d$. Also, $(B(W),\left\Vert \cdot \right\Vert )$ is a
Banach space wherein convergence is uniform. Therefore,
if we consider a Cauchy sequence $\{h_n \}$ in $B(W)$, then the
sequence $\{h_n \}$ converges uniformly to a function, say $h^*$,
so that $h^* \in B(W)$.

\vspace{.3cm}
We consider the operators $T_{i}: B(W) \rightarrow B(W)$
given by
\begin{equation}
T_{i}h_{i}(x)=\sup_{y\in D}\Big\{g(x,y)+G_{i}(x,y,h_{i}(\tau (x,y)))\Big\},
\label{3.3}
\end{equation}
for $h_{i}\in B(W)$, $x \in W$, for $i=1,2$; these mappings
are  well-defined if the functions $g$ and $G_{i}$ are bounded.
Also, denote
\begin{align*}
\Theta(h,k) =
\max \left\{\begin{array}{cc}
d(T_2h,T_2k),d(T_2h,T_1h),d(T_2k,T_1k), \frac{d(T_1h,T_2k)+d(T_1k,T_2h)}{2},\\
\frac{d(T_1h,T_2 h)d(T_1k,T_2k)}{1 + d(T_2 k, T_2 h)}, \frac{d(T_1 h, T_2 k)d(T_1 k, T_2h)}{1 + d(T_2 k, T_2 h)},
\frac{d(T_1 h, T_2 k)d(T_1 k,T_2 h)}{1 + d(T_1h, T_1k)}
  \end{array}
  \right\},
\end{align*}
for $h,k\in B(W)$.

\vspace{.3cm}
\begin{theorem}
\label{TA1} Let $T_i:B(W) \rightarrow B(W)$ be given by
\eqref{3.3}, for $i = 1,2$. Suppose that the following
hypotheses hold:
\begin{itemize}
\item[(1)] there exist $\tau \in \mathbb{R}^{+}$ and $\varphi \in \Phi$ such that
$$
\left\vert G_{1}(x,y,h(x))-G_{2}(x,y,k(x))\right\vert
 \leq  e^{-\tau} \varphi (\Theta(h,k))
$$
for all $x \in W$, $y \in D$; \item[(2)] $g: W \times D
\rightarrow \mathbb{R}$ and $G_i: W \times D \times \mathbb{R}
\rightarrow \mathbb{R}$ are bounded functions, for $i =1,2$;
\item[(3)] there exists a sequence $\{h_{n}\}$ in $B(W)$
and a function $h^* \in B(W)$ such that
\begin{equation*}
  \lim_{n\to\infty}T_1h_{n}=\lim_{n\to\infty}T_2h_{n}=T_1h^*;
\end{equation*}
\item[(4)] $T_1T_1h=T_1h$, whenever $T_1h=T_2h$, for some $h\in
B(W)$.
\end{itemize}
Then the system of functional equations \eqref{fe} has a bounded
solution.
\end{theorem}

\begin{proof}
By hypothesis (3), the pair $(T_1,T_2)$ shares the common limit
range property with respect to $T_1$. Now, let $\lambda $ be an
arbitrary positive number, $x\in W$ and $h_{1},h_{2}\in B(W)$.
Then there exist $y_{1},y_{2}\in D$ such that
\begin{align}
T_1h_{1}(x) &<g(x,y_{1})+G_1(x,y_{1},h_{1}(\tau
(x,y_{1})))+\lambda ,
\label{3.4} \\
T_2h_{2}(x) &<g(x,y_{2})+G_2(x,y_{2},h_{2}(\tau
(x,y_{2})))+\lambda ,
\label{3.5} \\
T_1h_{1}(x) &\geq g(x,y_{2})+G_1(x,y_{2},h_{1}(\tau (x,y_{2}))),
\label{3.6} \\
T_2h_{2}(x) &\geq g(x,y_{1})+G_2(x,y_{1},h_{2}(\tau (x,y_{1}))).
\label{3.7}
\end{align}
Next, by using \eqref{3.4} and \eqref{3.7},  we obtain
\begin{align*}
T_1h_{1}(x)-T_2h_{2}(x) &<G_1(x,y_{1},h_{1}(\tau
(x,y_{1})))-G_2(x,y_{1},h_{2}(\tau (x,y_{1})))+\lambda \\
&\leq \left\vert G_1(x,y_{1},h_{1}(\tau
(x,y_{1})))-G_2(x,y_{1},h_{2}(\tau
(x,y_{1})))\right\vert +\lambda \\
&\leq   e^{-\tau} \varphi(\Theta(h_{1},h_{2}))+\lambda
\end{align*}
and so we have
\begin{equation}
T_1h_{1}(x)-T_2h_{2}(x) < e^{-\tau} \varphi(\Theta(h_{1},h_{2}))+\lambda.
\label{3.8}
\end{equation}
Analogously, by using \eqref{3.5} and \eqref{3.6},  we get
\begin{equation}
T_2h_{2}(x)-T_1h_{1}(x) < e^{-\tau} \varphi(\Theta(h_{1},h_{2}))+\lambda
\label{3.9}
\end{equation}
Combining \eqref{3.8} and \eqref{3.9},  we get
\begin{equation*}
\left\vert T_1h_{1}(x)-T_2h_{2}(x)\right\vert < e^{-\tau}
\varphi(\Theta(h_{1},h_{2}))+\lambda,  
\end{equation*}
implying thereby
\begin{equation*}
d(T_1h_{1},T_2h_{2}) \leq e^{-\tau} \varphi(\Theta(h_{1},h_{2}))+\lambda.
\end{equation*}
Notice that, the last inequality does not depend on $x\in W$ and
$\lambda >0$ is taken arbitrarily, therefore we obtain
\begin{equation*}
d(T_1h_{1},T_2h_{2})\leq e^{-\tau} \varphi(\Theta(h_{1},h_{2})).
\end{equation*}
By passing to logarithms, we can write
\begin{align*}
\tau + \ln(d(T_1h_{1},T_2h_{2})) & \leq \ln(\varphi(\Theta(h_{1},h_{2}))).
\end{align*}
If we consider $F\in\mathcal{F}$ defined by $F(t) = \ln t$, for
each $t\in(0,+\infty)$, and put $f=T_1$, $T=T_2$, then all the hypotheses
of Theorem~\ref{theorem1} are satisfied for the pair $(f, T)$ and $p=1$.
Moreover, in view of the hypotheses (4), the pair $(T_1,T_2)$ is
occasionally coincidentally idempotent, so by using Theorem~\ref{theorem1}, the mapping $T_1$ and $T_2$
have a common fixed point, that is, the system of functional
equations \eqref{fe} has a bounded solution.
\end{proof}

\vspace{.4cm}
\subsection{Application to Volterra integral inclusions}
Here, we present yet another application of Theorem
\ref{theorem2}. This application is essentially inspired by \cite{Turko1}.

\vspace{.3cm}
We establish new results on the existence of solutions of
integral inclusion of the type
\begin{equation}
x(t) \in q(t) + \int_{0}^{\sigma(t)}  k(t, s)F(s, x(s))\,ds
\label{I-1}
\end{equation}
for $t \in J=[0,1]\subset \mathbb{R} $, where $\sigma: J \to J$, $q: J \to E$, $k: J \times
J \to \mathbb{R}$ are continuous and $F: J \times E \to C(E)$,
where $E$ is a Banach space with norm $\|\cdot\|_E$ and $C(E)$
denotes the class of all nonempty closed subsets of $E$.

\vspace{.3cm}
Let $C(J, E)$ be the space of all continuous $E$-valued
functions on $J$. Define a norm $\|\cdot\|$ on $C(J, E)$ by
$$
\|x\| = \sup_{ t \in J} \|x(t) \|_E .
$$

\begin{definition}
  A continuous function $a \in C(J, E)$ is called a lower solution of the integral inclusion (\ref{I-1}),  if it satisfies
$$a(t) \leq q(t)+ \int_{0}^{\sigma(t)} k(t, s)v_1(s)ds, ~{\rm for ~all}~ v_1 \in B(J, E)$$
 such that $v_1(t) \in F(t, a(t)) ~{\rm almost ~everywhere~}(a.e.)~{\rm for}~ t  \in J$, where $B(J, E)$  is the space of all $E$-valued Bochner-integrable functions on $J$. Similarly, a continuous function $b\in C(J, E)
$ is called an upper solution of the integral inclusion (\ref{I-1}), if it satisfies
$$b(t) \geq q(t)+ \int_{0}^{\sigma(t)} k(t, s)v_2(s)ds, ~{\rm for ~all}~ v_2 \in B(J, E)$$
such that $v_2(t) \in F(t, b(t)) ~a.e.~{\rm for}~ t \in J.$
\end{definition}

Notice that, all the solution lies between lower solution $`a$' as well upper solution $`b$'. We can denote the solution set as an interval $[a,b].$

\begin{definition}
  A continuous function $x : J \to E$ is said to be a solution of the integral inclusion (\ref{I-1}), if
$$x(t) = q(t) + \int_{0}^{\sigma(t)} k(t,s)v(s)\,ds $$
for some $v \in B(J, E)$ satisfying $v(t) \in F(t, x(t))$, for
all $t \in J$.
\end{definition}

\vspace{.3cm}

In what follows, we also need the following definitions:

\begin{definition}
A multi-valued mapping $F : J \to 2^E$ is said to be measurable if for
any $y \in E$, the function $t \mapsto d(y, F(t)) = \inf \{\|y -
x\| : x \in F(t)\}$ is measurable.
\end{definition}

\begin{definition}
A multi-valued mapping $ \beta  : J \times E \to 2^E$ is called
Carath\'eodory if

(i) $t \mapsto  (t, x)$ is measurable for each $x \in E$, and

(ii) $x \mapsto  (t, x)$ is upper semicontinuous almost everywhere
for $t \in J$.
\end{definition}

Denote
$$
\| F(t, x) \| = \sup \{\|u \|_E : u \in F(t, x)\}.
$$

\begin{definition}
A Carath\'eodory multi-mapping $F(t, x)$ is called
$L^1$-Carath\'eodory if for every real number $r > 0$, there exists
a function $h_r \in L^1(J,\mathbb{R})$  such that
$$
\|F(t, x) \| \leq  h_r (t) \text{ for almost every } t \in J
$$
and for all $x \in E$ with $\|x \|_E \leq r$.
\end{definition}

Denote
$$
S^1_F(x) = \big\{v \in B(J, E) : v(t) \in F(t, x(t))  ~a.e.~ t \in J \big\}.
$$

\vspace{.3cm}
\begin{lemma}\cite{Lasota}
If $diam(E) < \infty$ and $F : J \times E \to 2^E$ is
$L^1$-Carath\'eodory, then $S^1_F (x) \neq \emptyset$ for each $x
\in C(J, E)$.
\end{lemma}

\vspace{.3cm}
\begin{lemma}\cite{Turko1}
Let $E$ be a Banach space, $F$ a Carath\'eodory multi-mappping with
$S^1_F \neq \emptyset $ and $\mathcal{L} : L^1(J, E) \to C(J,
E)$ a continuous linear mapping. Then the operator
$$\mathcal{L} \circ S^1_F: C(J, E) \to 2^{C(J,E)}$$  is a closed
graph operator on $C(J, E) \times C(J, E)$.
\end{lemma}

\vspace{.3cm}
Let us list the following set of conditions:

\begin{itemize}

\item[$(H_0)$] The function $k(t, s)$ is continuous and
non-negative on $J \times J$ with
$$
e^{-\tau} = \sup_{ t, s \in J} k(t, s)
$$
for some $\tau \in \mathbb{R}^{+}$;

\item[$(H_1)$] The multi-valued mapping $F(t, x)$ is
Carath\'eodory;

\item[$(H_2)$] The multi-valued mapping $F(t, x)$ is increasing
in $x$ almost everywhere for $t \in  J$;

\item[$(H_3)$] There exist $\tau \in \mathbb{R}^{+}$ and $\varphi \in \Phi$ such that
$$
\left\vert F(s,x(s))-F(s,y(s))\right\vert \leq  e^{-\tau} \varphi(\Delta(x,y))
$$
for all $s \in J$, $x \in E$, where
\begin{eqnarray*}
\nonumber
\Delta(x,y) &=& { \alpha} |fx-fy| + \frac{ { \beta} \big[1+
|fx-Tx|\big]|fy-Ty| }{1+ |fx - fy|}+{\gamma} \big[|fx - Tx | + |fy - Ty|\big]\\
&+&\delta \big[{|fx - Ty|} + |fy - Tx|\big]
\end{eqnarray*}
${\text {with}}~ \alpha, \beta, \gamma,
 \delta \geq 0, \alpha + \beta + 2\gamma +2 \delta \leq 1;$

\item[$(H_4)$] $S^1_F(x) \neq \emptyset$ for each $x \in C(J,
E)$.
\end{itemize}

\vspace{.3cm}
\begin{theorem}
Suppose that the conditions $(H_0)$--$(H_4)$  hold. Then the integral
inclusion \eqref{I-1} has a solution in $[a, b]$ defined on $J$.
\end{theorem}

\begin{proof}
Let $X = C(J, E)$. Define a multi-valued mapping $T : [a, b]\subset X \to 2^X$ given by
$$
Tx = \biggl\{u\in [a,b] : u(t) = q(t) + \int_{0}^{\sigma(t)} k(t,
s)v(s)\,ds;~~ v \in S^1_F(x), \text{ for every } t \in [0, 1]
\biggr\}.
$$
Observe that $T$ is well-defined, as owing to $(H_4)$,
$S^1_F(x) \neq \emptyset$. To show that $T$ satisfies all
hypotheses of Theorem \ref{theorem2} defined on $[a, b]$.

\vspace{.3cm}
For all $\vartheta, \mu \in 2^X$ on $t \in J$ and making use of $(H_0)$ and
$(H_3)$, we have
\begin{align*}
\|\vartheta(t) - \mu(t) \|_E &= \biggl\| \int_{0}^{\sigma(t)} k(t,
s)v_1(s)\,ds - \int_{0}^{\sigma(t)} k(t, s)v_2(s)\,ds
\biggr\|_E \\
&\leq \int_{0}^{\sigma(t)} k(t, s)\, ds \| v_1(s) - v_2(s) \|_E\\
&\leq \sup_{ t, s \in J} k(t, s) \varphi(\Delta(v_1, v_2)) \text{ for } v_1, v_2
\in S^1_F(x).
\end{align*}
This implies that
\begin{align*}
\|\vartheta(t) - \mu(t) \|_{E} &\leq e^{-\tau} \varphi(\Delta(v_1, v_2)).
\end{align*}
for each $t \in J$. Passing to logarithms, we can write this as
\begin{align*}
\tau + \ln\|\vartheta(t) - \mu(t) \|_{E} & \leq \ln(\varphi(\Delta(v_{1},v_{2}))).
\end{align*}
If we consider $F\in\mathcal{F}$ defined by $F(z) = \ln z$, for
each $z\in(0,+\infty)$, we deduce that the operator $T$ satisfy
condition (\ref{equation2}) where $f$ is an identity mapping  and $p=1$. Also
$T$ is a closed mapping, using Theorem \ref{theorem2}, we conclude
that the given integral inclusion has a solution in $[a,b]$.
\end{proof}

\vspace{.3cm}
\noindent{\bf Competing interests:}
The authors declare that they have no competing interests.

\vspace{.3cm}
\noindent{\bf Authors’ contributions:}
All authors contributed equally and significantly in writing this article. All authors read and approved the final manuscript.

\end{document}